\documentclass{CSML}
\pdfoutput=1

\usepackage{hyperref}
\hypersetup{hidelinks}

\usepackage{lastpage}
\newcommand{\dOi}{13(3:13)2017}
\lmcsheading{}{1--\pageref{LastPage}}{}{}%
{Dec.~30, 2016}{Aug.~22, 2017}{}

\usepackage{amsthm}
\usepackage{amsmath}
\usepackage{amssymb}

\def\rd{\triangleright}
\def\red{\triangleright^*}

\def\a{\alpha}
\def\b{\beta}

\def\G{\Gamma}
\def\D{\Delta}

\def\d{\delta}

\def\l{\lambda}

\def\m{\mu}
\def\o{\omega}

\def\s{\sigma}

\def\sou{\underline}

\def\f{\rightarrow}

\def\v{\vdash}

\def\<{\langle}
\def\>{\rangle}
\def\F{\displaystyle\frac}
\def\eq{\simeq}

\newcommand{\cal}[1]{\ensuremath{\mathcal{#1}}}
\def\N{\ifmmode{\rm I\mkern-3.1mu
N\mkern0.5mu}\else{\rm I\kern-.16em
N\hskip0.5pt\ }\fi\relax} 

\def\B{\ifmmode{\rm \bot\mkern-10.1mu
\bot\mkern0.5mu}\else{\rm \bot\kern-.16em
\bot\hskip0.5pt\ }\fi\relax}

\theoremstyle{plain}
\newtheorem{theorem}[thm]{Theorem}
\newtheorem{lemma}[thm]{Lemma}

\newtheorem{corollary}[thm]{Corollary}
\theoremstyle{definition}
\newtheorem{definition}[thm]{Definition}

\begin{document}
\title[A revised completeness result for the simply typed $\lambda \m$-calculus]{A revised completeness result for the simply typed $\lambda \m$-calculus using realizability semantics}
\author{Karim Nour}
\address{LAMA - \'Equipe LIMD, Universit\'e Savoie Mont Blanc, Le Bourget du Lac, France}
\email{karim.nour@univ-smb.fr}
\author{Mohamad Ziadeh}

\begin{abstract}
In this paper, we define a new realizability semantics for the simply 
typed $\l\m$-calculus.  We show that if a term is typable, then it inhabits the interpretation of its type. 
We also prove  a completeness result of our realizability semantics using a particular term model.
This paper corrects some errors in \cite{ns3} by the first author and Saber.
\end{abstract}

\maketitle

\section{Introduction}

Since it was realised that the Curry-Howard isomorphism can be extended to the case of classical logic
(cf. \cite{gri,mur}), several calculi have appeared aiming to give an encoding of proofs formulated
either in classical natural deduction or in classical sequent calculus. One of them was the $\l\m$-calculus presented 
by Parigot in \cite{par0}, which stands very close in nature to the $\l$-calculus itself. 
It uses new kinds of variables, the $\m$-variables,
not active at the moment, but to which the current continuation can be passed over. Eliminating cuts with these new formulas
leads to the introduction of a new reduction rule: the so-called $\m$-rule. The result is a calculus, the $\l\m$-calculus 
 which is in relation with classical natural deduction. In addition, more simplification rules, for example 
the $\rho$- and $\theta$-rules, are defined by Parigot (cf. \cite{par1, par2}). Parigot showed that the $\l\m$-calculus is strongly 
normalizing in \cite{par3}, he gave a proof of the result with the help of the Tait-Girard reducibility method (cf. \cite{gir, tai}).
An arithmetical proof of the same result was presented by David and Nour in \cite{dn}.

The idea of the realizability semantics is to associate to each type a set
of terms which realize this type. Under this semantics, an atomic type 
is interpreted as a set of $\lambda$-terms saturated by a certain relation.  
Then, an arrow type receives the intuitive interpretation of a
functional space. For example, a term which realizes
the type ${\mathbb N}\rightarrow {\mathbb N}$ is a function from
${\mathbb N}$ to ${\mathbb N}$.  Realizability semantics has been a
powerful method for establishing the strong normalization of type
systems \`a la Tait and Girard (cf. \cite{tai, gir}). The realizability of a type system
enables one to also show the soundness of the system in the sense that
the interpretation of a type contains all the terms that have this
type. Soundness has been an important method for characterizing the
computational behaviour of typed terms through their types as has been
illuminative in the work of Krivine.

It is also interesting to find the class of types for which the
converse of soundness holds i.e., to find the types $A$ for which
the realizability interpretation contains exactly, in a certain sense,
the terms typable by $A$. This property is called completeness and
has not yet been studied for every type system.

Hindley was the first to study the completeness of a simple type system (cf. \cite{hin2, hin3, hin4}) 
and he showed that all the types of that system have the completeness
property.  Then, he generalised in~\cite{hin1} his completeness proof for an
intersection type system. In \cite{lab}, Farkh and Nour established
completeness for a class of types in Girard-Reynolds's system F known as the strictly positive types.
In \cite{fn1, fn2}, the same authors generalised the result of~\cite{lab} for the
larger class which includes all the positive types and also for the second
order functional arithmetic. Coquand established in \cite{coq} 
by a different method using Kripke models the completeness
for the simply typed $\lambda$-calculus. Finally Kamareddine and Nour 
established in \cite{kam} the soundness and completeness for a strict non linear intersection 
type system with a universal type.  

In \cite{ns1}, Nour and Saber adapted Parigot's method and established
a short semantical proof of the strong normalization of classical
natural deduction with disjunction as primitive. In general all the known semantical proofs of strong normalization use
a variant of the reducibility candidates based on a correctness
result, which has been important also for characterizing computational
behavior of some typed terms, as it was done in Krivine's works \cite{kri2}.
This inspired Nour and Saber  to define a general semantics for classical
natural deduction in \cite{ns2} and give such characterizations.
In \cite{ns3}, Nour and Saber proposed a realizability semantics for 
the simply typed $\l\m$-calculus and proved a completeness theorem. This semantics is inspired
by the strong normalization proof of Parigot's $\l\m$-calculus, which consists in 
rewriting each reducibility candidate as a double orthogonal.
In \cite{vbd}, van Bakel, Barbanera and de'Liguoro prove the completeness of the $\l\mu$-calculus with intersection 
types using filter models.

The semantics proposed in \cite{ns3} has several defaults. First of all, there is a mistake in the proof of
the correctness lemma (the case of the typing rule $(\mu)$ in the proof of Lemma 3.3) which is not correctable. 
This comes essentially from the permission to have $\mu$-variables in 
the sequence of terms by adopting Saurin's interpretation. This mistake makes the semantics 
less interesting even though the proof of completeness remains correct. 
Besides, although the statement of Lemma 4.2 of \cite{ns3} is correct, the proposed proof is false. 
The correction of these mistakes, mainly the first, needs the introduction of another realizability 
semantics which is completely different from the one proposed in \cite{ns3}.

In the present work we provide another realizability semantics for the simply typed $\l\m$-calculus, by 
changing the concept of saturation. For this, we add an important and an indispensable
modification to the notion of $\m$-saturation using more bottom sets. The saturation conditions
give a very satisfactory correctness result. The completeness model gives also more intuition about the 
models that we consider.

This paper is organized as follows. Section 2 is an introduction to the
simply typed $\l\m$-calculus. We also define  the semantics and
prove its correctness. Section 3  is devoted to the
completeness result.

\section {The simply typed $\l \m$-calculus}

\subsection{The syntax of the system}

In this paper, we use the $\l \m$-calculus \`a la de Groote which is more
expressive than Parigot's original version.
In \cite{dg}, de Groote has proposed a new version of the $\l \m$-calculus
by modifying its syntax. Namely, in the construction of terms the distinction between named and
unnamed terms has disappeared and the term forming rules has became more flexible:
a $\mu$-operator can be followed now by any kind of term (in the untyped version), not
necessarily by a term beginning with a $\mu$-variable. 

\begin{definition}[Terms and reductions]\leavevmode
\begin{enumerate}
\item Let $\mathcal{V}$ and $\mathcal{A}$ be two infinite sets of
disjoint alphabets for distinguishing $\l$-variables and 
$\mu$-variables. The $\l \m $-terms are given by the following
grammar:

\begin{center}
$ \mathcal{T}$$ \;\; := $$\;\; \mathcal{V}$$ \;\;  | \;\; \lambda$$\mathcal{V}$$.\mathcal{T}\;\;
| $$\;\; (\mathcal{T}\;\mathcal{T})\;\; |$$ \;\;\mu \mathcal{A}.\mathcal{T}$$ \;\;|\;\; (\mathcal{A}\;\mathcal{T})$
\end{center}

\item The complexity of a term is defined by
 $c(x) = 0$, $c(\l x.t) = c(\m \a.t) = c((\a \; t)) = c(t) +1$ and $c((u \; v)) = c(u) + c(v) + 1$.

\item The basic reduction rules are $\triangleright_\beta$ and $\triangleright_\m$.
\begin{itemize}
\item $(\lambda x.u \; v) \triangleright_\beta u[x:=v]$
\item $(\m \a.u\; v) \triangleright_\m \m \a.u[\a:=^*v]$

where  $u[\a:=^*v]$ is obtained from $u$ by replacing inductively each
subterm in the form $(\a\;w)$ in  $u$ by  $(\a\;(w\;v))$.
\end{itemize}

\item Let $\rd$ stand for the compatible closure of the union of the rules given above, and, as usual, 
$\red$ (resp. $\rd^+$) denotes the reflexive transitive (resp. transitive) closure
of $\rd$, and $\eq$ the equivalence relation induced by $\red$. We also
denote by $t \rd^n t'$ that $t \red t'$ and the length of this reduction 
(number of reduced redexes) is $n$. A term is said to be normal if it does not contain a redex.
A term is called normalizable if there is somme reduction sequence which terminates.
A term $t$ is called strongly normalizable, if there is no infinite
reduction sequence starting from $t$.
\end{enumerate}
\end{definition}

We find in the current  literature of the $\l \m$-calculus other simplification rules such as: 
$\rd_{\theta}$, $\rd_{\rho}$, $\rd_{\nu}$, $\rd_\eta$, $\rd_{\m'}, \ldots$
These rules allow to get less normal forms (see \cite{nou1, par1, par3, py}).
In this paper, we do not need these rules for our completeness result.

\begin{definition}[Types and typing rules]\leavevmode
\begin{enumerate}
\item Types are formulas of the propositional logic built from the
infinite set of propositional variables $\mathcal{P} = \{X, Y,Z,\dots  \}$
and a constant of type $\perp$, using the connective $\to$. 

\item The complexity of a type is defined by
$c(\perp) = c(X) = 0$ and
$c(A \f B) = c(A) + c(B) +1$.

\item Let $A_1,A_2,\dots  ,A_n,A$ be types, we denote the type $A_1 \to
(\dots  \to (A_n \to A)\dots  )$ by $A_1,\dots  , A_n\to A$.

\item Proofs are presented in  natural deduction system with several
conclusions, such that formulas in the left-hand-side of $\vdash$ are indexed by
$\l$-variables and those in right-hand-side of $\vdash$ are indexed by
$\m$-variables, except one which is indexed by a term.
Let $t$ be a $\l \m $-term, $A$ a type,
$\Gamma=\{x_i:A_i\}_{1 \le i \le n}$ and $\Delta =\{\a_j:B_j\}_{1 \le j
\le m}$ contexts, using the following rules, we will define ``$t$ typed with
type $ A$ in the contexts $\Gamma$ and $ \Delta$'' and we denote it $
\Gamma \vdash t:A\; ; \Delta.$

\begin{center}
 $\F{}{\Gamma\, \vdash x_i:A_i\,\, ; \, \Delta}{ax\;\;}\;\;\;$ $1 \leq i \leq n$
\end{center}

\begin{center}
$\F{\Gamma, x:A \vdash t:B;\Delta}{\Gamma \vdash \lambda x.t:A \to B;\Delta}\;\;{\to_i}
\;\;\;\;\;\;\;\;\;\;\;\;\;\;\;\;\;\;
\F{\Gamma \vdash u:A \to B;\Delta \quad \Gamma \vdash
v:A;\Delta}{\Gamma\vdash (u\; v):B;\Delta}\;\;{\to_e}$
\end{center}

\begin{center}
$\F{\Gamma\vdash t:\perp; \Delta, \a:A}{\Gamma \vdash\m \a. t:A; \Delta}\;\;{\m}
\;\;\;\;\;\;\;\;\;\;\;\;\;\;\;\;\;\;\;\;\;\;\;\;\;\;\;\;\;\;
\quad \F{\Gamma\vdash t:A;\Delta, \a:A}{\G \v (\a\; t):\perp; \D, \a:A}\;\;{\perp}$
\end{center}

We denote this typed system by $S_\m$.


\end{enumerate}
\end{definition}

We have the following results (for more details, see \cite{dn, par3, py}).

\begin{theorem}[Confluence result]
If  $t \red t_1$ and  $t \red t_2$, then there exists  $t_3$ such that  $t_1 \red t_3$ and  $t_2 \red t_3$.
\end{theorem}

\begin{theorem}[Subject reduction]
If  $\G \vdash t : A ; \Delta$ and  $t \red t'$ then $\G \vdash t' : A ; \Delta$.
\end{theorem}

\begin{theorem}[Strong normalization]\label{SN}
If  $\G \vdash t : A ; \Delta$, then $t$ is strongly  normalizable.
\end{theorem}

We need some specific definitions and notations.

\begin{definition}\leavevmode
\begin{enumerate}
\item Let $t$ be a term. The term $\vec{\l \m}.t$ denotes the term $t$ preceded 
by a sequence of $\lambda$ and $\mu$ abstractions.
\item Let $t$ be a term and $\bar{v}$ a finite sequence of terms $($the empty sequence is denoted by $\emptyset$$)$. 
The term $(t \, \bar{v})$ is defined by $(t \; \emptyset) = t$ and
$(t \; u \bar{u}) = ((t \; u) \; \bar{u})$.
\item  Let us recall that a term $t$ either has a head redex
(i.e. $t=\vec{\l \m}. (u \, \bar{v})$ where $u$ is a redex called the head redex), 
or is in head normal form
(i.e. $t=\vec{\l \m}.(x \;\bar{v})$ or $t=\vec{\l \m}.((\a \;u)\,\bar{v})$ where $x$ and $\a$ are variables called the head variable).  
\item The leftmost reduction (denoted by $\rd_l$) consits in reducing the redex nearest to the left of the term. 
We can also see it as an iteration of the head reduction: 
once we find the head normal form, we reduce the arguments of the head variable.
\end{enumerate}
\end{definition}

\begin{lemma}\label{PY1}
The leftmost reduction of a normalizing term terminates.
\end{lemma}

\begin{proof} 
See \cite{py}.
\end{proof} 

The previous lemma allows one to define a notion of length for normalizing terms.

\begin{definition}
Let $t$ be a normalizing term. We define $l(t)$ as the number of leftmost reductions 
needed to reach the normal form of $t$.
\end{definition}

We need to define the concept of simultaneous substitution to be able to present the correctness lemma.

\begin{definition}
Let $t,u_1,\dots  ,u_n$ be terms, $\bar{v}_1,\dots  ,\bar{v}_m$
finite sequences of terms, and $\s $ the simultaneous substitution 
$[(x_i:=u_i)_{1 \le i\le n};(\a_j:=^*\bar{v}_j)_{1 \le j \le m}]$ which is not an object of the syntax. 
Then $t\s$ is obtained from the term $t$ by replacing each $x_i$ by $u_i$ and replacing inductively each subterm of the form
$(\a_j \, u)$ in $t$ by $(\a_j (u\,\bar{v}_j))$. Formally $t \s$ is defined by induction on $c(t)$ as follows:
\begin{itemize}
\item If $t = x$ and $x \neq x_i$, then, $t\s = x$.
\item If $t = x_i$, then $t\s = u_i$.
\item If $t = \lambda x. u$, then we can assume that $x$ is a new $\lambda$-variable and $t\s = \lambda x. u\s$.
\item If $t = (u \, v)$, then $t = (u\s \, v\s)$.
\item If $t = \mu \alpha. u$, then we can assume that $\alpha$ is a new $\mu$-variable and $t\s = \mu \alpha. u\s$.
\item If $t = (\alpha \, u)$ and $\alpha \neq \alpha_j$, then $t\s = (\alpha \, u\s)$.
\item If $t = (\alpha_j \, u)$, then $t\s = (\alpha_j \, (\, u\s \, \bar{v}_j))$.
\end{itemize}
\end{definition}

\begin{lemma}\label{sub1}
If $t\red t'$ and $\sigma$ a simultaneous substitution, then $t\sigma \red t'\sigma$.
\end{lemma}

\begin{proof} 
It suffies to check the property for one step of reduction. Then we proceed by induction on $c(t)$.
\end{proof}

\subsection{The semantics of the system}

In this section, we define the realizability semantics and we prove its correctness lemma.
We begin by giving the definition of the saturation of sets of terms and then the operation $\leadsto$ 
between these sets which will serve to interpret the arrow on types.

\begin{definition}\leavevmode
\begin{enumerate}
\item We say that a set of terms ${\cal{S}}$ is saturated if
for all terms $u$ and $v$, if $v \red u$ and $ u\in {\cal{S}}$, then $v\in {\cal{S}}$.

\item Consider two sets of terms $\cal{K}$ and $\cal{L}$, we define a
new set of terms
$${\cal{K}} \leadsto {\cal{L}} =\{ t \in  {\cal{T}} \,\, / \,\, \forall \, u \in {\cal{K}} \; ; \; (t  \; u) \in {\cal{L}}\}.$$

\item We denote $\mathcal{T}^{<\omega}$ as the set of finite sequences of terms.  
Let ${\cal{L}}$ be a set of terms and $\cal{X} \subseteq
\cal{T}^{<\omega}$, then we define a new set of terms
$${\cal{X}} \leadsto {\cal{L}} =\{t\in  {\cal{T}} \,\, /\,\,  \forall \, \bar{u} \in {\cal{X}} \; ; \; (t\; \bar{u}) \in {\cal{L}} \}.$$
\end{enumerate}
\end{definition}

\begin{lemma} 
If $\cal{L}$ is a saturated set and $\cal{X} \subseteq
\cal{T}^{<\omega}$, then $\cal{X}\leadsto \cal{L}$ is also a
saturated one.
\end{lemma}

\begin{proof} 
Let $u$ and $v$ be terms such that  $v \red u$ and $u\in \cal{X}\leadsto \cal{L}$. 
Then $\forall \, \bar{w} \in {\cal{X}}$,  $(u\; \bar{w}) \in {\cal{L}}$ and $(v \; \bar{w}) \red (u \; \bar{w})$. 
Since $\cal{L}$ is a saturated set, we have $\forall \, \bar{w} \in {\cal{X}}$, $(v \; \bar{w}) \in \cal{L}$, 
thus $v\in \cal{X}\leadsto \cal{L}$.
\end{proof} 

Now, we are going to define the realizability model for the system $S_\mu$. For that, we need many bottom sets 
$(\B_i)_{i \in I}$ including a particular one denoted by $\B_0$. We also need several sets of $\mu$-variables 
$({\cal C}_i)_{i \in I}$ which allow to pass from bottoms to $\B_0$ and vice versa. We also allow that the models 
have particular sets $(R_j)_{j \in J}$ satisfying some properties. In order to obtain the completeness theorem, 
it is possible to take from the beginning $I = \mathbb{N}$ without considering the sets $(R_j)_{j \in J}$. However, 
the flexibility to have $I \subseteq \mathbb{N}$ and some sets $(R_j)_{j \in J}$ will allow to have more models. 
This will also allow the use of the correctness lemma in order to study the computational behaviour of some typed terms 
(cf. the example given at the end of this section, Theorem \ref{example}).

\begin{definition} \leavevmode
\begin{enumerate}
\item A  model ${\cal M}$ of $S_\m$ is defined by giving three subsets
$\<({\cal C}_i)_{i \in I},(\B_i)_{i \in I}, (R_j)_{j \in J}\>$ where:
\begin{itemize}
\item  $I,J$ are subsets of $\mathbb{N}$ such that $0\in I$,
\item $({\cal C}_i)_{i \in I}$ a sequence of disjoint infinite subsets of $\m$-variables,
\item $(\B_i)_{i \in I}$ and $(R_j)_{j \in J}$ sequences of non-empty saturated subsets of terms
\end{itemize}
such that
\begin{itemize}
\item  $\forall \, i \in I$, if $\a_i \in {\cal C}_i$ and $u \in \B_0$   
then $\m \a_i.u \in \B_i$ (i.e. if, for some $\a_i \in {\cal C}_i$, $u[\a := \a_i] \in \B_0$, then $\mu \a.u \in \B_i)$,
\item $\forall \, i \in I$, if $\a_i \in {\cal C}_i$ and $u_i \in \B_i$, 
then $(\a_i \, u_i) \in  \B_0$,
\item $\forall \, j \in J$, $\exists \, i \in I$, $\exists \, {\cal X}_j \subseteq \mathcal{T}^{<\o}$, such that
$R_j= {\cal X}_j  \leadsto \B_i$.
\end{itemize}

\item If ${\cal M} = \<({\cal C}_i)_{i \in I},(\B_i)_{i \in I},(R_j)_{j \in J} \>$  is a model of $S_\m$,
we denote by $\vert {\cal M} \vert$ the smallest set containing the sets $\B_i$ and $R_j$ and 
closed under the constructor $\leadsto$.
\end{enumerate}
\end{definition}

We will see further that we did not need to get the subsets $(R_j)_{j\in J}$ to have the correcteness lemma. 
We allow a model to have such sets to enrich the concept of model and give the possibility 
of interpreting the types with more options.

Now we are going to prove that every element of a realizability model can be written as the orthogonal of a set of sequence terms. 
This property is essential to interpret the $\mu$-variables and announce the generalized correctness lemma (Lemma \ref{adq}). 
The difficulty here with respect to the semantics proposed in \cite{ns1, ns2, ns3} is the presence of several bottoms. 
As for the orthogonal of an element of a model, we will choose the convenient bottom which has the smallest index (Lemma \ref{orth2}).

\begin{lemma} \label{orth1}
Let ${\cal M} = \<({\cal C}_i)_{i \in I},(\B_i)_{i \in I},(R_j)_{j \in J} \>$ 
be a model and ${\cal{G}} \in \vert\mathcal{M}\vert$. 
There exists a set
${\cal{X_G}} \subseteq \mathcal{T}^{<\o} $ and $i \in I$ such that ${\cal G} = {\cal{X_G}} \leadsto \B_i$.
\end{lemma}

\begin{proof} 
By induction on $\cal{G}$.
\begin{itemize}
\item If ${\cal{G}}=\B_i$, take  $\cal{X_G} = \{\phi\}$.
\item If ${\cal{G}}=R_j$, then, by definition of $R_j$, ${\cal G} = {\cal{X}}_j \leadsto \B_i$ and we take ${\cal X_G} =  {\cal X}_j$.
\item If ${\cal{G}}={\cal{G}}_1 \leadsto {\cal{G}}_2$, then, by induction
hypothesis, ${\cal{G}}_2={\cal{X}}_{{\cal{G}}_2}\leadsto \B_i$
where ${\cal{X}}_{{\cal{G}}_2} \subseteq \mathcal{T}^{<\o}$, and
take ${\cal{X_G}}=\{u\bar{v}$ / $u\in {\cal{G}}_1$ and $\bar{v}\in
{\cal{X}}_{{\cal{G}}_2}\}$. \qedhere
\end{itemize}
\end{proof}

\begin{definition} 
If ${\cal M} = \<({\cal C}_i)_{i \in I},(\B_i)_{i \in I}, (R_j)_{j \in J}\>$ 
is a model and ${\cal{G}} \in \vert \mathcal{M}\vert$, let
\begin{itemize}
\item $w({\cal{G}})$ the smallest integer  $i$ such that 
${\cal{G} = X} \leadsto \B_i$ for some ${\cal X} \subseteq \mathcal{T}^{<\o}$,
\item  ${\cal{G}}^\perp =\bigcup\{{\cal{X}} \subseteq \mathcal{T}^{<\o} \,\, / \,\, {\cal{G=
X}} \leadsto \B_{w({\cal{G}})}\}$. 
\end{itemize}
\end{definition}

\begin{lemma} \label{orth2}
Let ${\cal M} = \<({\cal C}_i)_{i \in I},(\B_i)_{i \in I},(R_j)_{j \in J} \>$ 
be a model and ${\cal{G}} \in \vert\mathcal{M}\vert$, then 
${\cal G} = {\cal G}^\perp \leadsto \B_{w({\cal{G}})}$.
\end{lemma}

\begin{proof} \leavevmode
\begin{itemize}[leftmargin=8mm]
\item[$\subseteq$ :] Let $t \in {\cal G}$. If $\bar{u} \in {\cal G}^\perp$, 
then $\bar{u} \in {\cal{X}} \subseteq \mathcal{T}^{<\o}$ and ${\cal{G=
X}} \leadsto \B_{w({\cal{G}})}$, thus $(t \; \bar{u}) \in \B_{w({\cal{G}})}$. 
Therefore $t \in {\cal G}^\perp \leadsto \B_{w({\cal{G}})}$.

\item[$\supseteq$ :] By Lemma \ref{orth1}, we have ${\cal G} = {\cal X} \leadsto \B_{w({\cal{G}})}$ 
for some ${\cal X} \subseteq \mathcal{T}^{<\o}$, then ${\cal X} \subseteq  {\cal G}^\perp$. 
Let $t \in {\cal G}^\perp \leadsto \B_{w({\cal{G}})}$. We have $\forall \, \bar{u} \in {\cal X}$, $\bar{u} \in {\cal G}^\perp$, 
then $(t \; \bar{u}) \in \B_{w({\cal{G}})}$. Therefore $t \in {\cal G}$.
\qedhere
\end{itemize}
\end{proof} 

We can now interpret the types in a model and also give the definition of the general interpretation of a type.

\begin{definition} \leavevmode
\begin{enumerate}
\item Let ${\cal M} = \<({\cal C}_i)_{i \in I},(\B_i)_{i \in I}, (R_j)_{j \in J}\>$ be a model.
\begin{enumerate}
\item An $\mathcal{M}$-interpretation $\cal{I}$ is a
function $X \mapsto {\cal{I}}(X)$ from the set of propositional
variables $\mathcal{P}$ to $\vert\mathcal{M}\vert$ which we extend for any
formula as follows: ${\cal{I}}(\perp)=\B_0$ and
${\cal{I}}(A \to B)= {\cal{I}}(A) \leadsto {\cal{I}}(B)$.
\item For any type $A$, we denote $\vert A \vert_\mathcal{M} =\bigcap
\{ {\cal{I}}(A) \,\, / \,\, {\cal{I}}$ an $\mathcal{M}$-interpretation$\}$ the interpretation of $A$ in the model $\mathcal{M}$.
\end{enumerate}
\item For any type $A$, we denote $|A|= \bigcap\{\vert A \vert_\mathcal{M} \,\, / \,\, \mathcal{M}$ a model$\}$ 
the general interpretation of~$A$.
\end{enumerate}
\end{definition}

The correctness lemma is sort of a validation of the notion of models that we considered. 
It states that a term of a type $A$ is within the general interpretation of $A$. As for the semantics defined in \cite{ns3}, 
the proof of the correctness lemma is false and the mistake is difficult to find. It is the case of the typing rule $(\mu)$ 
which is problematic. The mistake comes from the permission to have $\mu$-variables in the sequence of terms by adopting Saurin's 
interpretation.

\begin{lemma} [General correctness lemma] \label{adq} 
Let  ${\cal M} = \<({\cal C}_i)_{i \in I},(\B_i)_{i \in I}, (R_j)_{j \in J}\>$ be a model,
${\cal{I}}$ an $\mathcal{M}$-interpretation,
$\Gamma =\{x_k : A_k\}_{\substack {1\le k\le n}}$, 
$\Delta =\{\a_r: B_r\}_{\substack {1\le r\le m}}$ such that 
$\a_r \in {\cal{C}}_{ w({\cal{I}}(B_r))}$ $(1\le r\le m)$, $u_k \in \;{\cal{I}}(A_k)$ $(1\le k\le n)$, 
$\bar{v_r} \in\; ({\cal{I}}(B_r))^\perp$ $(1\le r\le m)$) and \\
$\sigma =[(x_k:=u_k)_{1 \le k \le n};(\a_r:=^*\bar{v}_r)_{1\le r\le m}]$.
If $\Gamma \vdash u:A \,\,\,;\,\Delta$, then $u\sigma \in {\cal{I}}(A)$.
\end{lemma}

\begin{proof} 
By induction on the  derivation, we consider the last rule used.

\begin{itemize}[leftmargin=8mm]

\item[$ax$:] In this case $u=x_k$ and $A=A_k$, then $u\s =u_k \in {\cal{I}}(A)$.

\item[$\to_i$:] In this case $u= \l x.v$ and $A =B \to C$ such that
$\G ,x:B \v v:C \; ;\D$. Let $w \in {\cal{I}}(B)$ and $\d
=\s+[x:=w]$, by induction hypothesis, $v\d \in {\cal{I}}(C)$. Since
$(\l x.v\s \; w) \red v\d$, then $(\l x.v\s \; w) \in {\cal{I}}(C)$. Therefore $\l x.v \s \in
{\cal{I}}(B) \leadsto {\cal{I}}(C)$ and finally $u\s \in {\cal{I}}(A)$.

\item[$\to_e$:] In this case $u =(u_1 \, u_2)$, $\G \v u_1:B \to A \;
  ;\D$ and $\G \v u_2:B\; ;\D$.  By induction hypothesis, $u_1\s
  \in {\cal{I}}(B) \leadsto {\cal{I}}(A)$ and $u_2\s \in
      {\cal{I}}(B)$, therefore $(u_1\s \, u_2\s) \in {\cal{I}}(A)$, i.e. $u\s \in {\cal{I}}(A)$.

\item[$\m$:] In this case $u =\m \a. v$ and  $\G \v v: \perp \; ;\a : A , \D$. 
We can assume that $\a$ is a new variable which belongs to $\mathcal{C}_{w({\cal{I}}(A))}$ 
 (there is always such a
variable because the sets $\mathcal{C}_i$ are infinite). Let $\bar{v}\in
({\cal{I}}(A))^{\perp}$ and $\d =\s +[\a:=^*\bar{v}]$. By  induction hypothesis,
$v\d \in {\B_0}$, then $\m \a. v\d \in \B_{w({\cal{I}}(A))}$. Since
$(\m \a. v \s \; \bar{v}) \red \m \a. v \d$, then, $(\m \a. v \s \; \bar{v}) \in \B_{w({\cal{I}}(A))}$.
We deduce that for all $\bar{v}\in ({\cal{I}}(A))^{\perp}$, $(\m \a. v \s \; \bar{v}) \in \B_{w({\cal{I}}(A))}$,
therefore $\m \a. v \s \in {\cal{I}}(A)$, i.e. $u\s \in {\cal{I}}(A)$.

\item[$\bot$:] In this case $u =(\a_r \; v)$, $A = \perp$, $\D = \a_r : B_r , \D'$ and $\G \v v: B_r \; ; \D$. 
By induction hypothesis, $v\s \in {\cal{I}}(B_r)$, hence $(v\s \;\bar{v_r}) \in
\B_{w({{\cal{I}}(B_r))}}$, therefore $(\a_r \;(v \s \;\bar{v_r})) \in \B_0$, i.e. $u\s \in {\cal{I}}(A)$.
\qedhere
\end{itemize}
\end{proof}

We can now state and prove the correctness lemma.

\begin{corollary}[Correctness lemma] \label{closed}
Let $A$ be a type and $t$ a closed term. 
If $\vdash t:A$, then~$t \in \vert A \vert$.
\end{corollary}

\begin{proof} 
Let $\mathcal{M}$ be a model and ${\cal{I}}$ an
$\mathcal{M}$-interpretation. Since $\vdash t:A$, then, by the general correctness lemma, 
$t \in {\cal{I}}(A)$. This is true for any model
$\mathcal{M}$ and for any $\mathcal{M}$-interpretation ${\cal{I}}$, therefore
$t \in |A|$.
\end{proof}

According to the cases $\to_i$ and $\m$ of the proof of Lemma \ref{adq}, 
we observe that the saturation conditions can be weakened as follows. 
We do not need the saturation by expansion but only the saturation by weak-head expansion.

{\it A set of terms ${\cal{S}}$ is saturated if:
\begin{itemize}
\item for all terms $u,v,\overline{w}$, 
if $(u[x := v] \; \overline{w}) \in {\cal{S}}$, then $(\l x.u \; v\overline{w})\in {\cal{S}}$.
\item for all terms $u,\overline{v}$, if $\mu \a.u[\a :=^* \overline{v}] \in {\cal{S}}$, 
then $(\mu \a.u \; \overline{v})\in {\cal{S}}$.
\end{itemize}
}

We take an example from \cite{ns3} to show that it is sometimes useful not to have in a model an infinite number 
of bottoms as well as the relevance of the sets $(R_j)_{j \in J}$.
Let $e =\lambda x. \mu \a. x$, then we have $\vdash e : \perp \to X$ and for all finite sequence
of terms $t \, \bar{u}$, $(e \; t\, \bar{u}) \red \m \a. t$. 
We will prove this general result.

\begin{theorem}\label{example} Let $E$ be a closed term of type $\perp \to X$, then, for each
finite sequence of distinct $\l$-variables $x \, \bar{y}$, $(E \;x \, \bar{y}) \red \sou{\m}. x$ where 
$\sou{\m}. x$ denote the variable $x$ preceded by $\mu$-abstractions and $\mu$-applications.
\end{theorem}

\begin{proof} 
Let $x \, \bar{y}$ be a finite sequence of distinct
$\l$-variables, $\mathcal{C}_0=\mathcal{A}$, 
$\B_0=\{t \in {\cal T}\,\, / \,\, t\red \sou{\m}. x \}$, 
${\cal{R}}=\{\bar{y}\}\leadsto \B_0$, $\mathcal{M}=\<\mathcal{C}_0,\B_0,{\cal{R}}\>$ and
${\cal{I}}$ the interpretation such that
${\cal{I}}(X)={\cal{R}}$. By Corollary \ref{closed}, $E \in {\cal{I}}
(\perp \to X) = \B_0 \leadsto  (\{\bar{y}\}\leadsto \B_0)$. Since $x \in \B_0$ and $\bar{y} \in \{\bar{y}\}$, 
we have $(E\;x \, \bar{y}) \in 
\B_0$, and finally $(E \;x\, \bar{y}) \red \sou {\m}. x$.
\end{proof}

The model that we considered in the previous proof only contains $\B_0$ 
and the set $\cal{R}$ is necessary to find the computational behavior of the term $E$.

\section{The completeness result}

Before presenting our completeness result, we need some definitions and some technical results.

\subsection{Some results}

Lemma \ref{ynorm} is very intuitive. It states that, for a $\l$-variable $y$, if the term $(t \, y)$ is normalizable, 
then $t$ is also normalizable. The proof of this result in $\l$-calculus is very simple because a reduction 
of $(t \, y)$ will automatically give a reduction of $t$ (the variable $y$ interacts only when $t$ reduces 
to a $\l$-abstraction $\l x. t'$ and we get in this case $t'[x:=y]$). The situation in $\l\mu$-calculus is completly different. 
If the term $t$ is reduced to a $\mu$-abstraction $\mu \a. t'$ and we reduce the term $(\mu \a. t' \, y)$, we get 
$\mu \a.t'[\a: =^*y]$, the variable $y$ can be found in several sub-terms of $t'$ and may also create other redex. Hence the need for a detailed and 
comprehensive proof of Lemma \ref{ynorm}. Note also that the proof of this result in \cite{ns3} is not correct.

\begin{lemma}\label{ynorm}
Let $t$ be a term and $y$ a $\l$-variable.  If $(t \, y)$ is normalizable, then $t$ is also normalizable.
\end{lemma}

\begin{proof} 
By induction on the pair $(l,c)$ (for the lexicographical order) where $l = l((t \, y))$ and $c = c(t)$.

\begin{enumerate}

\item If $t$ does not begin with $\l$ or $\m$, then $(t \, y) \rd_l (t' \, y)$ where
$t  \rd_l t' $. By induction hypothesis ($l$ decreases), $t'$ is normalizable, then $t$ is also normalizable.

\item If $t$ begins with a $\m$, then $t=\m \a. \vec{\l \m}. (u \, \bar{u})$.

\begin{itemize}
\item If $u = (\l x.v \, w)$, then 

$(t \, y) \rd_l \m \a.\vec{\l \m}. (\l x.v[\a:=^*y] \, w[\a =^*y] \; \bar{u}[\a:=^*y]) \rd_l$

$\m \a. \vec{\l \m}. (v[\a:=^*y][x:=w[\a:=^*y]] \; \bar{u}[\a:=^*y]) = $

$\m \a. \vec{\l \m}. (v[x:=w][\a:=^*y] \; \bar{u}[\a:=^*y]) = t''$.

On the other hand, $t \rd_l \m \a. \vec{\l \m}. (v[x:=w] \; \bar{u}) = t'$ and $(t' \, y) \rd_l t''$,  
then  $(t' \, y)$ is normalizable, therefore, by induction hypothesis ($l$ decreases), $t'$ is normalizable, 
thus $t$ is also normalizable.

\item If $u = (\m \b.v \, w)$, then 

$(t \, y) \rd_l \m \a.\vec{\l \m}. (\m \b.v[\a:=^*y] \, w[\a:=^*y] \; \bar{u}[\a:=^*y]) \rd_l$

$\m \a. \vec{\l \m}. (\m \b.v[\a:=^*y][\b:=^*w[\a:=^*y]] \; \bar{u}[\a:=^*y]) = $

$\m \a. \vec{\l \m}. (\m \b.v[\b:=^*w][\a:=^*y] \; \bar{u}[\a:=^*y]) = t''$.

On the other hand, $t \rd_l \m \a. \vec{\l \m}. (\m \b.v[\b:=^*w] \; \bar{u}) = t'$ and $(t' \, y) \rd_l t''$,  
then  $(t' \, y)$ is normalizable, therefore, by induction hypothesis ($l$ decreases), $t'$ is normalizable, 
thus $t$ is also normalizable.

\item If $u = x$, then 

$(t \, y) \rd_l \m \a.\vec{\l \m}. (x \; u_1[\a:=^*y]\dots u_n[\a:=^*y])$.

The term $u_i[\a:=^*y]$ is normalizable and $(\m \a.u_i \, y) \rd_l \m \a. u_i[\a:=^*y]$, then $(\m \a.u_i \, y)$ 
is normalizable and, by induction hypothesis ($l$ does not decrease but $c$ decreases), 
$\m \a.u_i$  is normalizable, thus $u_i$ is also normalizable. We deduce that $t$ is also normalizable.

\item If $u = (\b \, u')$, then 

$(t \, y) \rd_l \m \a.\vec{\l \m}. ((\b \, u'[\a:=^*y]) \; u_1[\a:=^*y]\dots u_n[\a:=^*y])$. 

As in the previous case, we prove that the terms $u',u_1,\dots  ,u_n$ are normalizable, then  $t$ is also normalizable.

\item If $u = (\a \, u')$, then 

$(t \, y) \rd_l \m \a.\vec{\l \m}. ((\a \, (u'[\a:=^*y] \, y) ) \; u_1[\a:=^*y]\dots u_n[\a:=^*y])$. 

As in the previous case, we prove that the terms $(u' \,y)$,$u_1,\dots  ,u_n$ are normalizable, then, by induction hypothesis 
($l$ does not decrease but $c$ decreases),  $u',u_1,\dots  ,u_n$ are normalizable, then  $t$ is also normalizable.

\end{itemize}

\item If $t$ begins with $\l$, we do the same proof
\qedhere
\end{enumerate}
\end{proof} 

In the technical results that we prove in this section, a particular kind of redex appears 
(the argument of $\l$-abstraction or $\mu$-abstraction is a variable).
It is therefore helpful to understand what is happening by reducing these redexes.
Lemmas \ref{yredex} and \ref{ytype2} will be used in the proof of the completeness theorem.

\begin{definition}
Let $y$ be a $\l$-variable. 
A $y$-redex is a $\b$-redex or a $\m$-redex having $y$ as an argument 
i.e. a term of the form $(\l x.t \, y)$ or $(\m \a. t \, y)$. 
\end{definition}


\begin{lemma}\label{yredex}
Let $t$ be a normal term and $y$ a $\l$-variable. 
If $(t \, y) \rd^+t'$, then 
each redex in $t'$ is a $y$-redex preceded by a $\m$-variable.
\end{lemma}

\begin{proof}
By induction the number $n$ of steps to reduce $(t \, y)$ to $t'$.
\begin{enumerate}
\item If $n = 1$, then $t$ is of the form $\l x.u$ or $\m \a.u$.
\begin{itemize}
\item If $t = \l x.u$, then $t' = u[x:=y]$ which is normal.
\item If $t = \m \a.u$, then $t' = \m \a.u[\a:=^*y]$. 
Since $u$ is normal, every new redex of $u[\a:=^*y]$ is of the form $(v \, y)$ preceded by $\a$.
\end{itemize}
\item If $(t \, y) \rd^{n+1}t'$, then $(t \, y) \rd^{n}t'' \rd t'$. By induction,  
each redex in $t''$ is a $y$-redex preceded by a $\m$-variable. We examine the reduction $t'' \rd t'$.
\begin{itemize}
\item If the reduced $y$-redex of $t''$ is of the form $(\l x.u \, y)$, then 
the redexes of $t'$ are the same in $t''$ except the redex that we contracted.
\item If the reduced $y$-redex of $t''$ is of the form $(\m \b.u \, y)$,  then 
$y$ has evolved from being an argument to this redex, to become an argument for each subterm, of this redex, 
preceded by $\b$ ; so a new $y$-redex preceded by a $\m$-variable has been created and the initial redex 
could not get a new argument as it is locked by $\b$.
\qedhere
\end{itemize}
\end{enumerate}
\end{proof}




Lemma \ref{ytype2} will allow us to rebuild the typing in the proof of the completeness theorem. 
To prove it, we need the following lemma.

\begin{lemma}\label{ytype1}
Let $t$ be a term, $y$ a $\l$-variable and $\a$ a $\m$-variable.
If $\G, y : B \v {t[\a :=^* y]:A} ;$ $ \a : C, \D$, then $\G, y : B \v t : A ; \a: B \f C, \D$.
\end{lemma}

\begin{proof}
By induction on $c(t)$. 

\begin{enumerate}
\item If $t = x$, the result is trivial.
\item If $t = \l x.u$, then $A = E \f F$ and $\G, y : B , x : E \v u[\a :=^* y] : F ; \a : C, \D$.
By induction hypothesis, we have  $\G, y : B , x : E \v u: F ; \a : B \f C, \D$ and 
$\G, y : B \v t : A ; \a: B \f C, \D$.
\item If $t = (u \, v)$, then $\G, y : B \v u[\a :=^* y] : E \f A ; \a : C, \D$ and $\G, y : B \v v[\a :=^* y] : E ; \a : C, \D$.
 By induction hypothesis, we have  $\G, y : B \v u : E \f A ; \a : B \f C, \D$ and $\G, y : B \v v : E ; \a : B \f C, \D$, thus
$\G, y : B \v t : A ; \a: B \f C, \D$.
\item If $t = \m \b. u$, then $\G, y : B \v u[\a :=^* y] : \perp ; \a : C, \b : A , \D$. 
 By induction hypothesis, we have $\G, y : B \v u: \perp ; \a : B \f C, \b : A , \D$ and
$\G, y : B \v t : A ; \a: B \f C, \D$.
\item If $t = (\b \, u)$ and $\b \neq \a$, then $A = \perp$, $\D = \b : A , \D'$ and $\G, y : B \v u[\a :=^* y] : A ; \a : C, \b : A , \D'$.
 By induction hypothesis, we have $\G, y : B \v u: A ; \a : B \f C, \b : A , \D'$ and
$\G, y : B \v t : A ; \a: B \f C, \D$.
\item If $t = (\a \, u)$, then $A = \perp$ and $\G, y : B \v (\a \, (u[\a :=^* y] \, y)) : \perp ; \a : C , \D$, thus
 $\G, y : B \v (u[\a :=^* y] \, y) : C ; \a : C , \D$ and $\G, y : B \v u[\a :=^* y]: B \f C ; \a : C , \D$.
  By induction hypothesis, we have $\G, y : B \v u: B \f C ; \a : B \f C , \D$ and $\G, y : B \v t : A ; \a: B \f C, \D$.
  \qedhere
\end{enumerate}
\end{proof}

The next lemma is ``subject expansion'' restricted to $y$-redex.

\begin{lemma}\label{ytype2}
Let $t$ be a term,  $y$ a $\l$-variable and $t'$ a term obtained by reducing some $y$-redexes from $t$.  
If $\G, y : B \v t' : A ; \D$, then $\G , y : B \v t : A ; \D$.
\end{lemma}

\begin{proof}
It is enough to prove the result for a one reduction. By induction on $c(t)$.
\begin{enumerate}
\item If $t = \l x.u$ and $t' = \l x.u'$ with $u \red u'$, then $A = C \f D$ and $\G , y : B , x : C \v u' : D ; \D$.
By induction hypothesis, we have $\G , y : B , x : C \v u : D ; \D$ and $\G , y : B \v t : A ; \D$.
\item If $t = \m \a.u$  and $t' = \m \a. u'$ with $u \red u'$, then
$\G , y : B , x : C \v u' : \perp ; \a : A, \D$
By induction hypothesis, we have $\G , y : B , x : C \v u : \perp ; \a : A, \D$ and $\G , y : B \v t : A ; \D$.
\item If $t = (u \, v)$ and $t' = (u' \, v)$ with $u \red u'$, then $\G , y : B \v u' : C \f A ; \D$ and
$\G , y : B \v v : C ; \D$. By induction hypothesis, we have $\G , y : B \v u : C \f A ; \D$ and 
$\G , y : B \v t : A ; \D$.
\item If $t = (u \, v)$ and $t' = (u \, v')$ with $v \red v'$, we make the same proof of the previous case.
\item If $t = (\l x. u \, y)$ and $t' = u[x:=y]$, then $\G , y : B \v u[x:=y] : A ; \D$. Let $y'$ be a new $\l$-variable, then
$\G , y : B , y' : B \v u[x:=y'] : A ; \D$. We deduce
$\G , y : B \v \l y'.u[x:=y'] : B \f A ; \D$, thus 
 $\G , y : B \v (\l y'.u[x:=y'] \, y) : A ; \D$ and $\G , y : B \v t : A ; \D$.
\item If $t = (\m \a. u \, y)$ and $t' = \m \a. u[\a :=^* y]$, then $\G , y : B \v u[\a :=^* y] : \perp ; \a : A , \D$.
By Lemma \ref{ytype1}, we have $\G , y : B \v u : \perp ; \a : B \f A , \D$, then 
$\G , y : B \v \m \a.u : B \f A , \D$ and $\G , y : B \v t : A ; \D$.
\qedhere
\end{enumerate}
\end{proof}

\subsection{Completeness model}

We will now prove that if $t$ is in the general interpretation of a type $A$, then $t$ has the type $A$. 
For this, we will construct a particular term model $\mathbb{M}$ in which, we will get the equivalence between 
``having the type $A$'' and ``being within the type $A$'' (see Lemma \ref{fatiguant}). 
The construction of this model looks like the constructions of the completeness models of the papers 
\cite{hin1, hin2, hin3, hin4, lab, fn1, fn2, kam, ns3}. We start with enumerating infinite sets of variables
(which will be parameters of this model), then enumerating sets of types and finally fixing two 
infinite contexts (by associating enumerated types to enumerated variables) in which the terms will be typed. 
This will allow to define the bottoms and the variable sets associated to the models.
Note that in this completeness model we don't need the sets $(R_j)_{j\in J}$.

\begin{definition}\label{compl}\leavevmode
\begin{enumerate}
\item Let $\mathbb{V}_1 =\{x_i$ / $i\in \mathbb{N}\}$ (resp.  $\mathbb{V}_2 = \{\a_i$ /
$i \in \mathbb{N}\}$ be an enumeration of an infinite set of $\l$-variables (resp.
$\mu$-variables). We put $\mathbb{V} = \mathbb{V}_1 \bigcup \mathbb{V}_2$.

\item Let $\mathbb{T}_1=\{A_i$ / $i\in \mathbb{N}\}$ and $\mathbb{T}_2=\{B_i$ / $i\in \mathbb{N}\}$ 
be enumerations of all types.


\item We define $\mathbb{G} =\{x_i : A_i$ / $i\in \mathbb{N}\} $ and $\mathbb{D}
=\{\a_i : B_i $ / $i\in \mathbb{N}\}$.

\item Let $u$ be a  term, such that $FV(u) \subseteq   \mathbb{V}$, the
 contexts $\mathbb{G}_u$ $($resp. $ \mathbb{D}_u$$)$ are defined as the
 restrictions of $\mathbb{G}$ $($resp. $\mathbb{D}$$)$ at the declarations
 containing the variables of  $FV(u)$.

\item The notation $\mathbb{G} \vdash u:C ;\,\,\mathbb{D}$ means that
$\mathbb{G}_u \vdash u:C ;\,\,\mathbb{D}_u $, we denote $\mathbb{G}
\vdash^* u:C ;\,\,\mathbb{D}$ iff there exists a term $u'$, such that
$u \red u'$ and $\mathbb{G} \vdash u':C ;\,\,\mathbb{D}$.

\item Let  $\mathbb{P} = \{ X_i$ / $i\in \mathbb{N}\}$ be an enumeration of 
$\{\perp\} \cup \mathcal{P}$. We assume that $X_0 = \perp$. 

\item For each $i \in \mathbb{N}$, 
let $\B_i = \{ t$ / $\mathbb{G} \vdash^* t: X_i ;\,\,\mathbb{D}\}$
and

$\mathbb{C}_i= \{\a \in \mathbb{V}_2 $ / $(\a  :  X_i) \in \mathbb{D}\}$.
\end{enumerate}
\end{definition}

\begin{lemma}
$\mathbb{M} =\<{(\mathbb{C}_i)}_{i \in \mathbb{N}},(\B_i)_{i \in\mathbb{N}},\emptyset\>$ is a model for $S_\m$.
\end{lemma}

\begin{proof} 
It is easy to show that $(\B_i)_{i \in \mathbb{N}}$ is a sequence of saturated subsets of terms.

\begin{itemize}
\item  $\forall \, i \in \mathbb{N}$, if $\a \in \mathbb{C}_i$ and $u \in \B_0$, 
then $(\a  :  X_i) \in \mathbb{D}$ and $\mathbb{G} \vdash^* u : \perp ;\,\,\mathbb{D}$, therefore
$\mathbb{G} \vdash^* \m \a.u : X_i ;\,\,\mathbb{D}$, thus $\m \a.u \in \B_i$.
\item $\forall \, i \in \mathbb{N}$, if $\a \in \mathbb{C}_i$ and $u \in \B_i$, 
then $(\a  :  X_i) \in \mathbb{D}$ and $\mathbb{G} \vdash^* u : X_i ;\,\,\mathbb{D}$, therefore
$\mathbb{G} \vdash^* (\a \, u) : \perp ;\,\,\mathbb{D}$, thus $(\a \, u) \in  \B_0$.
\qedhere
\end{itemize}
\end{proof}

Observe that the model $\mathbb{M}$ is parameterized by the two
infinite sets of variables and the enumerations. We need just these
infinite sets of variables and not all the variables. This is an
important remark since it will serve us in the proof of compleness theorem
(Theorem \ref{comp}).

\begin{definition}
We define the $\mathbb{M}$-interpretation $\mathbb{I}$ as follows
 $\forall \, i \in \mathbb{N}$,  $\mathbb{I}(X_i)= \B_i$.
\end{definition}

The following lemma is the generalized version of the completeness theorem. It states the equivalence between 
``a term $t$ is of type $A$ in the fixed contexts'' and ``a term $t$ belongs to the interpretation of the type $A$ 
in the model $\mathbb{M}$''. The proof is done by simultaneous induction and uses the technical lemmas from the beginning 
of this section.

\begin{lemma}\label{fatiguant}
Let $A$ be a type and $t$ a term.
\begin{enumerate}
\item If $\mathbb{G} \, \vdash^* t:A \; ;\mathbb{D}$, then $t \in \mathbb{I}(A)$.
\item If $t \in \mathbb{I}(A)$, then $\mathbb{G} \, \vdash^* t:A \; ; \mathbb{D}$.
\end{enumerate}
\end{lemma}

\begin{proof} By  a simultaneous induction on $c(A)$.

\begin{enumerate}
\item

\begin{enumerate}
\item If $A=X$ or $\perp$, the result is immediate from the definition of $\mathbb{I}$.
\item If  $A=B \to C$, then $t \red
  t'$ such that: $\mathbb{G} \, \vdash t': B \to C \; ;
  \mathbb{D}$. Let $u \in \mathbb{I}(B)$. By  induction
  hypothesis $(2)$, we have $\mathbb{G} \, \vdash^* u:B \; ;
  \mathbb{D}$, this implies that $u \red u'$ and $\mathbb{G} \, \vdash
  u': B \; ; \mathbb{D}$. Hence $\mathbb{G} \, \vdash (t'\;u'): C \; ;
  \mathbb{D}$, so, by the fact that $(t\;u) \red (t'\;u')$, we
  have $\mathbb{G} \, \vdash^* (t\;u):C \; ; \mathbb{D}$, then, by
  induction hypothesis $(1)$, $(t\;u) \in \mathbb{I}(C)$. Therefore
  $t \in \mathbb{I}(B \to C)$.
\end{enumerate}

\item

\begin{enumerate}
\item If $A=X$ or $\perp$, the result is immediate from the definition of $\mathbb{I}$.
\item If $A=B \to C$ and $t \in \mathbb{I}(B) \leadsto \mathbb{I}(C)$, let $y$ be a $\l$-variable 
such 
$(y:B) \in \mathbb{G}$. We have $y:B
\v y:B$, hence, by induction hypothesis
  $(1)$, $y \in \mathbb{I}(B)$, then, $(t\; y) \in
  \mathbb{I}(C)$. By induction hypothesis $(2)$, $\mathbb{G} \,
  \vdash^* (t\;y) :C \; ; \mathbb{D}$, then $(t\; y) \red u$ such
  that $\mathbb{G} \, \vdash u:C \; ; \mathbb{D}$ and, by the
 Lemma \ref{ynorm}, $t$ is a normalizable term. Let $t'$ (resp. $u'$) be the normal
form of $t$ (resp. of $u$). So $(t' \, y) \red u'$, by Lemma \ref{yredex}, $u'$ is obtained by reducing some $y$-redexes from 
$(t' \, y)$, then, by Lemma \ref{ytype2}, 
$\mathbb{G} \, \vdash (t' \, y) : C \; ; \mathbb{D}$, thus $\mathbb{G} \, \vdash t' :B \f C \; ; \mathbb{D}$ and
$\mathbb{G} \, \vdash^* t:B \f C \; ; \mathbb{D}$,
\qedhere
\end{enumerate}
\end{enumerate}
\end{proof}

Note that the item 1 of Lemma \ref{fatiguant} can not be deduced from the correctness lemma. 
This comes from the presence of contexts $\mathbb{G}$ and $\mathbb{D}$ to type a contractum of term $t$.

We can now state and prove the completeness theorem.

\begin{theorem}[Completeness theorem]\label{comp} 
Let $A$ be a type and $t$ a term. 
We have $t \in \vert A \vert$ iff
there exists a closed term $t'$ such that $t\red t'$ and $\v t':A$.
\end{theorem}

\begin{proof}\leavevmode
\begin{enumerate}[leftmargin=7mm]
\item[$\Leftarrow$:] 
By Corollary \ref{closed}, $t' \in \vert A \vert$, then, $t \in \vert A \vert$ because $\vert A \vert$ is saturated.

\item[$\Rightarrow$:] We consider an infinite set of $\l$ and $\m$ variables $\mathbb{V}$ which 
contains none of the free variables of $t$, then from this set
we build the completeness model as described in Definition
\ref{compl}. If $t\in \vert A \vert$, then $t \in \mathbb{I}(A)$,
hence by $(2)$ of Lemma \ref{fatiguant} and by the fact that
$FV(t') \subseteq FV(t)$, we have $t \red t'$ and $\vdash t':A$.
\qedhere
\end{enumerate}
\end{proof}

Here are some direct and unexpected consequences of the completeness theorem. 

\begin{corollary}
 Let $A$  be a  type and $t$ a term.
\begin{enumerate}
\item If $t \in \vert A \vert$, then $t$ is normalizable.
\item If $t \in \vert A \vert$, then there exists a closed term $t'$
such that $t \eq t'$.
\item The set $\vert A \vert$ is closed under equivalence. 
\end{enumerate}
\end{corollary}

\begin{proof} 
$(1)$ and $(2)$ are  direct consequences of Theorems \ref{SN} and \ref{comp}.  $(3)$
can be deduced from Theorem \ref{comp} and Lemma \ref{adq}.
\end{proof}

\section{Future work} 

Throughout this work we have seen that the propositional types of the
system $S_\m$ are complete for a realizability semantics. 
Two questions will be interesting to study.

\begin{enumerate}
\item The models that we have considered in this paper are sufficient to get correctness and completeness results: 
the saturation conditions that we have imposed in these models allow to have these two results. 
It will be interesting to understand more this kind of models: 
for example, build models with more bottoms to study the computational behaviour of some typed terms.

\item What about the types of the second order typed $\l\m$-calculus?
We know that, for the system ${\cal F}$, the $\forall^+$-types (types
with positive quantifiers) are complete for a realizability semantics
(see \cite{fn1, lab}). But for the classical system
${\cal{F}}$, we cannot easily generalize this result. 
We think we need to add more restrictions
on the positions of $\forall$ in the $\forall^+$-types to obtain the
smallest class of types that we suppose can be proved to be complete.





\end{enumerate}
\smallskip

\subsection*{Acknowledgements}
We wish to thank P. Batty\'anyi, N. Bernard and the anonymous referees 
for the numerous corrections and suggestions.

\end{document}